\documentclass[reqno, 10pt, draft]{article}
\usepackage[cp1251]{inputenc}
\usepackage{ amsfonts, amssymb, amsmath, amsthm}
\usepackage[ukrainian, russian, english, french]{babel}

\newtheorem{theorem}{Theorem}

\renewcommand{\normalsize}{\fontsize{14}{20}\selectfont}
\begin{document}
\selectlanguage{english} \thispagestyle{empty}
 \pagestyle{myheadings}              



\vskip 2mm

 \begin{center}
 {\LARGE \bf  Asymptotically best possible Lebesque-type inequalities 
for the Fourier sums on  sets of  generalized Poisson integrals }
 \end{center}

 \begin{center}
\noindent {\rm A.\,S. SERDYUK$^{(1)}$, T.\,A. STEPANYUK$^{(1,2)}$}
 \end{center}

 \begin{center}
{ \it $^{(1)}$ Institute of Mathematics of Ukrainian National Academy of Sciences, 3, Tereshchenkivska st., 01601, Kyiv-4, Ukraine, serdyuk@imath.kiev.ua;\\
$^{(2)}$ Johann Radon Institute for Computational and Applied Mathematics (RICAM)
Austrian Academy of Sciences, Altenbergerstr. 69 4040, Linz, Austria;  Institute of Mathematics of NAS of Ukraine, 3, Tereshchenkivska st., 01601, Kyiv-4, Ukraine, tania$_{-}$stepaniuk@ukr.net}
 \end{center}

\vskip 3.5mm

{\rm  \large A\,b\,s\,t\,r\,a\,c\,t  \ In this paper we establish Lebesgue-type inequalities for $2\pi$-periodic functions $f$, which are defined by generalized Poisson integrals of the functions $\varphi$ from $L_{p}$, $1\leq p< \infty$. In these inequalities uniform norms of deviations of Fourier sums 
$\| f-S_{n-1} \|_{C}$ are expressed via best approximations $E_{n}(\varphi)_{L_{p}}$ of functions $\varphi$   by trigonometric polynomials in the metric of space $L_{p}$. We show that obtained estimates are asymptotically best possible. 
\hfill}

{\rm  \large K\,e\,y\,w\,o\,r\,d\,s  \ 
Lebesgue-type inequalities, Fourier sums,  generalized Poisson integrals, best approximations by trigonometric polynomials

\hfill}

\thispagestyle{empty} \normalsize \vskip 3.5mm


\section{Introduction}\label{intro}

Let $L_{p}$,
$1\leq p<\infty$, be the space of $2\pi$--periodic functions $f$ summable to the power $p$
on  $[0,2\pi)$, in which
the norm is given by the formula
$\|f\|_{p}=\Big(\int\limits_{0}^{2\pi}|f(t)|^{p}dt\Big)^{\frac{1}{p}}$; $L_{\infty}$ be the space of measurable and essentially bounded   $2\pi$--periodic functions  $f$ with the norm
$\|f\|_{\infty}=\mathop{\rm{ess}\sup}\limits_{t}|f(t)|$; $C$ be the space of continuous $2\pi$--periodic functions  $f$, in which the norm is specified by the equality
 ${\|f\|_{C}=\max\limits_{t}|f(t)|}$.

Denote by $C^{\alpha,r}_{\beta}L_{p}, \ \alpha>0, \ r>0, \ \beta\in\mathbb{R}, \ 1\leq p\leq\infty,$ the set of all  $2\pi$--periodic functions, representable for  all
$x\in\mathbb{R}$ as convolutions of the form (see, e.g., \cite[p.~133]{Stepanets1})
\begin{equation}\label{conv}
f(x)=\frac{a_{0}}{2}+\frac{1}{\pi}\int\limits_{-\pi}^{\pi}P_{\alpha,r,\beta}(x-t)\varphi(t)dt,
\ a_{0}\in\mathbb{R}, \  \ \varphi\perp1, 
\end{equation}
where $\varphi\in L_{p}$ and $P_{\alpha,r,\beta}(t)$ are fixed generated kernels

\begin{equation}\label{kernel}
P_{\alpha,r,\beta}(t)=\sum\limits_{k=1}^{\infty}e^{-\alpha k^{r}}\cos
\big(kt-\frac{\beta\pi}{2}\big), \ \  \ \alpha,r>0, \ \  \beta\in
    \mathbb{R}.
\end{equation}
The kernels  $P_{\alpha,r,\beta}$ of the form \eqref{kernel} are called generalized Poisson kernels. For $r=1$ and $\beta=0$ the kernels $P_{\alpha,r,\beta}$ are usual Poisson kernels of harmonic functions.

If the functions $f$ and $\varphi$ are related by the equality \eqref{conv}, then function $f$ in this equality is called generalized Poisson integral of the function $\varphi$ and is denoted by
 $\mathcal{J}^{\alpha,r}_{\beta}(\varphi)(f(\cdot))=\mathcal{J}^{\alpha,r}_{\beta}(\varphi,\cdot)$. The function $\varphi$ in equality (\ref{conv}) is called as generalized derivative of the function $f$ and is denoted by $f^{\alpha,r}_{\beta}$ ($\varphi(\cdot)=f^{\alpha,r}_{\beta}(\cdot)$).

The set of functions $f$ from $C^{\alpha,r}_{\beta}L_{p}$, $1\leq p\leq \infty$, such that 
 $ f^{\alpha,r}_{\beta}\in B_{p}^{0}$, where
\begin{equation*}
B_{p}^{0}=\left\{\varphi: \ ||\varphi||_{p}\leq 1, \  \varphi\perp1\right\},
\end{equation*}
we will denote by $C^{\alpha,r}_{\beta,p}$.

Let $\tau_{2n-1}$ be the space of all trigonometric polynomials of degree at most $n-1$ and let $E_{n}(f)_{L_{p}}$ be the best approximation of the function $f\in L_{p}$ in the metric of space $L_{p}$, $1\leq p\leq\infty$, by the trigonometric polynomials $t_{n-1}$ of degree $n-1$, i.e.,
\begin{equation*}
E_{n}(f)_{L_{p}}=\inf\limits_{t_{n-1}\in \tau_{2n-1}}\|f-t_{n-1}\|_{p}.
\end{equation*}
Analogously, by $E_{n}(f)_{C}$ we denote the best uniform approximation of the function $f$ from $C$ by trigonometric polynomials of order $n-1$, i.e., 
\begin{equation*}
E_{n}(f)_{C}=\inf\limits_{t_{n-1}\in \tau_{2n-1}}\|f-t_{n-1}\|_{C}.
\end{equation*}

Let  $\rho_{n}(f;x)$ be the following quantity
\begin{equation}\label{rhoF}
\rho_{n}(f;x):=f(x)-S_{n-1}(f;x),
\end{equation}
 where $S_{n-1}(f;\cdot)$ are the partial Fourier sums of order  $n-1$ of a function $f$.
 
 Least upper bounds of the quantity $\|\rho_{n}(f;\cdot)\|_{C}$  over the classes $C^{\alpha,r}_{\beta,p}$, we denote by $ {\mathcal E}_{n}(C^{\alpha,r}_{\beta,p})_{C}$, i.e.,
 \begin{equation}\label{sum}
 {\mathcal E}_{n}(C^{\alpha,r}_{\beta,p})_{C}=\sup\limits_{f\in
C^{\alpha,r}_{\beta,p}}\|\rho_{n}(f;\cdot)\|_{C},  \ \ \ r>0, \ \alpha>0, \ 1\leq p \leq \infty.
  \end{equation}
  
  Asymptotic behaviour of the quantities $ {\mathcal E}_{n}(C^{\alpha,r}_{\beta,p})_{C}$ of the form (\ref{sum}) was studied in \cite{Stepanets1}--\cite{SerdyukStepanyuk2018}.
  
  In \cite{Stepanets1989N4}--\cite{SerdyukStepanyuk2018Bul}  the analogs of the Lebesque inequalities for functions $f\in C^{\alpha,r}_{\beta}L_{p}$ have been found in the case $r\in(0,1)$ and $p=\infty$, and also in the case  $r\geq1$ and $1\leq p\leq\infty$, where the estimates for the deviations $\|f(\cdot)-S_{n-1}(f;\cdot)\|_{C}$ are expressed in terms of the best approximations $E_{n}(f^{\alpha,r}_{\beta})_{L_{p}}$. Namely, in \cite{Stepanets1989N4} it was proved that for arbitrary $f\in C^{\alpha,r}_{\beta}$, $r\in(0,1)$, $\beta\in\mathbb{R}$, the following inequality holds
   \begin{equation}\label{Stepanets}
 \|f(\cdot)-S_{n-1}(f;\cdot)\|_{C}\leq \Big(\frac{4}{\pi^{2}}\ln n^{1-r}+\mathcal{O}(1)\Big)e^{-\alpha n^{r}}E_{n}(f^{\alpha,r}_{\beta})_{C},
  \end{equation}
  where $\mathcal{O}(1)$ is a quantity uniformly bounded with respect to $n$, $\beta$ and \linebreak ${f\in C^{\alpha,r}_{\beta}C}$.
  It was also shown that for any function $f\in C^{\alpha,r}_{\beta}C$  and for every $n\in \mathbb{N}$ one can find a function $\mathcal{F}\cdot)=\mathcal{F}f;n;\cdot)$ in the set
 $ C^{\alpha,r}_{\beta}C$, such that
  $E_{n}(\mathcal{F}^{\alpha,r}_{\beta})_{C}=E_{n}(f^{\alpha,r}_{\beta})_{C}$ and for this function the relation \eqref{Stepanets} becomes an equality.

  The present paper is a continuation of \cite{Stepanets1989N4}--\cite{SerdyukStepanyuk2018Bul}, and is devoted to obtain asymptotically best possible analogs of Lebesgue-type inequalities on the sets $C^{\alpha,r}_{\beta}L_{p}$, $r\in(0,1)$ and $p\in[1,\infty)$. This case was not considered yet.

It should be also noticed, that asymptotically best possible Lebesgue inequalities on classes of generalized Poisson integrals  $C^{\alpha,r}_{\beta}L_{p}$ for $r\in(0,1)$,  $p=\infty$  and $r\geq 1$, $1\leq p\leq\infty$ also were established for approximations by Lagrange trigonometric interpolation polynomials with uniform distribution of interpolation nodes (see, e.g., \cite{StepanetsSerdyuk2000}--\cite{StepanetsSerdyuk2012}).
  
\section{Main results}\label{intro}  
  Let us formulate the results of the paper.

By $F(a,b;c;d)$ we denote Gauss hypergeometric function
\begin{equation}\label{hypergeom}
F(a,b;c;z)=1+\sum\limits_{k=1}^{\infty}\frac{(a)_{k}(b)_{k}}{(c)_{k}}\frac{z^{k}}{k!},
\end{equation}
\begin{equation*}
(x)_{k}:=x(x+1)(x+2)...(x+k-1).
\end{equation*}

For arbitrary  $\alpha>0$, $r\in(0,1)$ and $1\leq p<\infty$ we denote by $n_0=n_0(\alpha,r,p)$ the smallest integer $n$ such that
\begin{equation}\label{n_p}
 \frac{1}{\alpha r}\frac{1}{n^{r}}+\frac{\alpha r p}{n^{1-r}}\leq{\left\{\begin{array}{cc}
 \frac{1}{14},  & p=1, \\
\frac{1}{(3\pi)^3}\cdot\frac{p-1}{p}, & 1< p<\infty.
  \end{array} \right.}
\end{equation}

The following theorem takes place.

\begin{theorem}\label{theorem1}
Let $0<r<1$,  $\alpha>0$, $\beta\in\mathbb{R}$ and $n\in \mathbb{N}$. Then in the case $1< p<\infty$ for any function
 $f\in C^{\alpha,r}_{\beta}L_{p}$ and $n\geq n_0(\alpha,r,p)$, the following inequality holds
\begin{equation*}
\|f(\cdot)-S_{n-1}(f;\cdot)\|_{C}\leq e^{-\alpha n^{r}}n^{\frac{1-r}{p}}\bigg(\frac{\|\cos t\|_{p'}}{\pi^{1+\frac{1}{p'}}(\alpha r)^{\frac{1}{p}}}F^{\frac{1}{p'}}\Big(\frac{1}{2}, \frac{3-p'}{2}; \frac{3}{2}; 1\Big)
\end{equation*}
\begin{equation}\label{Theorem1Ineq1}
+
\gamma_{n,p}\Big(\Big(1+\frac{(\alpha r)^{\frac{p'-1}{p}}}{p'-1}\Big) \frac{1}{n^{\frac{1-r}{p}}}+\frac{(p)^{\frac{1}{p'}}}{(\alpha r)^{1+\frac{1}{p}}}\frac{1}{n^{r}}\Big)\bigg)E_{n}(f^{\alpha,r}_{\beta})_{L_{p}}, \ \ \frac{1}{p}+\frac{1}{p'}=1,
\end{equation}
where $F(a,b;c;d)$ is Gauss hypergeometric function.

Moreover, for any function  $f\in C^{\alpha,r}_{\beta}L_{p}$ one can find a function $\\mathcal{F}(x)=\mathcal{F}f;n;x)$, such that $E_{n}(\mathcal{F}^{\alpha,r}_{\beta})_{L_p}=E_{n}(f^{\alpha,r}_{\beta})_{L_p}$ and the following equality holds
 \begin{align}\label{Theorem1Eq1}
&\|\mathcal{F}\cdot)-S_{n-1}(F;\cdot)\|_{C} \notag \\
&= e^{-\alpha n^{r}}n^{\frac{1-r}{p}}\bigg(\frac{\|\cos t\|_{p'}}{\pi^{1+\frac{1}{p'}}(\alpha r)^{\frac{1}{p}}}F^{\frac{1}{p'}}\Big(\frac{1}{2}, \frac{3-p'}{2}; \frac{3}{2}; 1\Big)\notag \\
&+
\gamma_{n,p}\Big(\Big(1+\frac{(\alpha r)^{\frac{p'-1}{p}}}{p'-1}\Big) \frac{1}{n^{\frac{1-r}{p}}}+\frac{(p)^{\frac{1}{p'}}}{(\alpha r)^{1+\frac{1}{p}}}\frac{1}{n^{r}}\Big)\bigg)E_{n}(f^{\alpha,r}_{\beta})_{L_{p}}, \ \ \frac{1}{p}+\frac{1}{p'}=1.
\end{align}

In \eqref{Theorem1Ineq1}  and \eqref{Theorem1Eq1}   the quantity ${\gamma_{n,p}=\gamma_{n,p}(\alpha,r,\beta)}$ is such that ${|\gamma_{n,p}|\leq(14\pi)^{2}}$.
\end{theorem}

\begin{proof}[Proof of Theorem~\ref{theorem1}]
Let us prove at the beginning the inequality \eqref{Theorem1Ineq1} .

Let $f\in C^{\alpha,r}_{\beta}L_{p}$, $1\leq p\leq \infty$. Then, at every point $x\in \mathbb{R}$ the following integral representation is true:
\begin{equation}\label{repr}
\rho_{n}(f;x)=f(x)-S_{n-1}(f;x)=
\frac{1}{\pi}\int\limits_{-\pi}^{\pi}f^{\alpha,r}_{\beta}(t)P_{\alpha,r,\beta}^{(n)}(x-t)dt,
\end{equation}
  where
 \begin{equation}\label{kernelN}
P_{\alpha,r,\beta}^{(n)}(t):=
\sum\limits_{k=n}^{\infty}e^{-\alpha k^{r}}\cos\Big(kt-\frac{\beta\pi}{2}\Big),  \ 0<r<1, \ \alpha>0, \ \beta\in\mathbb{R}.
\end{equation}
 
The function $P_{\alpha,r,\beta}^{(n)}(t)$ is orthogonal to any trigonometric polynomial $t_{n-1}$ of degree not greater than $n-1$. Hence, for any polynomial $t_{n-1}\in\tau_{2n-1}$  we obtain
\begin{equation}\label{for1}
\rho_{n}(f;x)=
\frac{1}{\pi}\int\limits_{-\pi}^{\pi}\delta_{n}(t)P_{\alpha,r,\beta}^{(n)}(x-t)dt,
\end{equation}
where
\begin{equation}\label{delta}
\delta_{n}(x)=\delta_{n}(\alpha,r,\beta;x):=f^{\alpha,r}_{\beta}(x)-t_{n-1}(x).
\end{equation}

Further we choose the polynomial $t_{n-1}^{*}$ of the best approximation of the function $f^{\alpha,r}_{\beta}$ in the space $L_{p}$, i.e., such that
\begin{equation*}
\| f^{\alpha,r}_{\beta}-t^{*}_{n-1}\|_{p}=E_{n}(f^{\alpha,r}_{\beta})_{L_{p}}, \ \ 1\leq p\leq\infty,
\end{equation*}
to play the role of $t_{n-1}$ in (\ref{for1}). Thus, by using the inequality
\begin{equation}\label{HolderIneq}
\bigg\|\int\limits_{-\pi}^{\pi}K(t-u)\varphi(u)du \bigg\|_{C}\leq \|K\|_{p'}\|\varphi\|_{p},
\end{equation}
\begin{equation*}
\varphi\in L_{p}, \ \ K\in L_{p'}, \ \ 1\leq p\leq\infty, \ \ \frac{1}{p}+\frac{1}{p'}=1
\end{equation*}
(see, e.g., \cite[p. 43]{Korn}), we get
\begin{equation}\label{for2}
\|f(\cdot)-S_{n-1}(f;\cdot)\|_{C}\leq
\frac{1}{\pi}\|P_{\alpha,r,\beta}^{(n)}\|_{p'}E_{n}(f^{\alpha,r}_{\beta})_{L_{p}}.
\end{equation}

It follows from the paper  \cite{SerdyukStepanyuk2017} (see, e.g., also \cite{SerdyukStepanyuk2016} and \cite{SerdyukStepanyuk2018})
 for arbitrary $r\in(0,1)$, $\alpha>0$, $\beta\in\mathbb{R}$, $1< p<\infty$, $\frac{1}{p}+\frac{1}{p'}=1$, $n\in\mathbb{N}$ and $n\geq n_0(\alpha,r,p)$
 the following  estimate holds
\begin{align}\label{normKern}
\frac{1}{\pi}\|P_{\alpha,r,\beta}^{(n)} \|_{p'}
&=e^{-\alpha n^{r}}n^{\frac{1-r}{p}}\left(\frac{\|\cos t\|_{p'}}{\pi^{1+\frac{1}{p'}}(\alpha r)^{\frac{1}{p}}}\left(  \int\limits_{0}^{\frac{\pi n^{1-r}}{\alpha r}}\frac{dt}{(t^{2}+1)^{\frac{p'}{2}}} \right)^{\frac{1}{p'}} \right.\notag \\
&
\left.+
\gamma_{n,p}^{(1)}\left(\frac{1}{(\alpha r)^{1+\frac{1}{p}}} \left(  \int\limits_{0}^{\frac{\pi n^{1-r}}{\alpha r}}\frac{dt}{(t^{2}+1)^{\frac{p'}{2}}} \right)^{\frac{1}{p'}}\frac{1}{n^{r}}+\frac{1}{n^{\frac{1-r}{p}}}\right)\right),
\end{align}
where $\frac{1}{p}+\frac{1}{p'}=1$ and the quantity ${\gamma_{n,p}^{(1)}=\gamma_{n,p}^{(1)}(\alpha,r,\beta)}$ satisfies the inequality  ${|\gamma_{n,p}^{(1)}|\leq(14\pi)^{2}}$.

In \cite{SerdyukStepanyuk2016} and \cite{SerdyukStepanyuk2017} it was mentioned that formula (\ref{normKern}) also holds, if in its second part instead $\frac{1}{\pi}\|P_{\alpha,r,\beta}^{(n)} \|_{p'}$ to put   $\frac{1}{\pi}\inf\limits_{\lambda\in\mathbb{R}}\|P_{\alpha,r,\beta}^{(n)} -\lambda\|_{p'}$  or
  $\sup\limits_{h\in\mathbb{R}}\frac{1}{2\pi}\|P_{\alpha,r,\beta}^{(n)} (t+h)-P_{\alpha,r,\beta}^{(n)} (t)\|_{p'}$  

Formula (106) from  \cite{SerdyukStepanyuk2018} gives the following estimate
\begin{equation}\label{form106}
\left(  \int\limits_{0}^{\frac{\pi n^{1-r}}{\alpha r}}\frac{dt}{(t^{2}+1)^{\frac{p'}{2}}} \right)^{\frac{1}{p'}}=
\left(  \int\limits_{0}^{\infty}\frac{dt}{(t^{2}+1)^{\frac{p'}{2}}} \right)^{\frac{1}{p'}}+ 
\frac{\Theta_{\alpha,r,p,n}^{(1)}}{p'-1}\Big(\frac{\alpha r}{\pi n^{1-r}} \Big)^{p'-1}, \ \ |\Theta_{\alpha,r,p,n}^{(1)}|<2.
\end{equation}

In the work \cite{SerdyukStepanyuk2017} (see formula (27)) it was shown, that for arbitrary \linebreak ${1<p'<\infty}$ the following equality takes place
\begin{equation}\label{eqHypergeom}
\left(  \int\limits_{0}^{\infty}\frac{dt}{(t^{2}+1)^{\frac{p'}{2}}} \right)^{\frac{1}{p'}}=
F^{\frac{1}{p'}}\Big(\frac{1}{2},\frac{3-p'}{2};\frac{3}{2};1 \Big).
\end{equation}

Taking into account the following estimate
\begin{equation}\label{ineq1}
\left(  \int\limits_{0}^{\frac{\pi n^{1-r}}{\alpha r}}\frac{dt}{(t^{2}+1)^{\frac{p'}{2}}} \right)^{\frac{1}{p'}}\leq\left(  \int\limits_{0}^{\infty}\frac{dt}{(t^{2}+1)^{\frac{p'}{2}}} \right)^{\frac{1}{p'}}<
\left( 1+ \int\limits_{1}^{\infty}\frac{dt}{t^{p'}} \right)^{\frac{1}{p'}}<(p)^{\frac{1}{p'}},
\end{equation}
formulas (\ref{normKern})--(\ref{ineq1}) imply that for $n\geq n_{0}(\alpha,r,p)$, $1<p<\infty$,  $\frac{1}{p}+\frac{1}{p'}=1$,
\begin{align}\label{normKern1}
\frac{1}{\pi}\big\|P_{\alpha,r,\beta}^{(n)} \big\|_{p'}&= e^{-\alpha n^{r}}n^{\frac{1-r}{p}}\bigg(\frac{\|\cos t\|_{p'}}{\pi^{1+\frac{1}{p'}}(\alpha r)^{\frac{1}{p}}}F^{\frac{1}{p'}}\Big(\frac{1}{2}, \frac{3-p'}{2}; \frac{3}{2}; 1\Big)
\notag \\
&+
\gamma_{n,p}^{(1)}\Big(\frac{1}{p'-1}\frac{(\alpha r)^{\frac{p'-1}{p}}}{ n^{(1-r)(p'-1)}}+\frac{p^{\frac{1}{p'}}}{(\alpha r)^{1+\frac{1}{p}}}\frac{1}{n^{r}}+\frac{1}{n^{\frac{1-r}{p}}}\Big)\bigg) \notag \\
&= e^{-\alpha n^{r}}n^{\frac{1-r}{p}}\bigg(\frac{\|\cos t\|_{p'}}{\pi^{1+\frac{1}{p'}}(\alpha r)^{\frac{1}{p}}}F^{\frac{1}{p'}}\Big(\frac{1}{2}, \frac{3-p'}{2}; \frac{3}{2}; 1\Big) \notag \\
&+
\gamma_{n,p}^{(2)}\Big(
\Big(1+\frac{(\alpha r)^{\frac{p'-1}{p}}}{p'-1}\Big)\frac{1}{ n^{\frac{1-r}{p}}}+\frac{p^{\frac{1}{p'}}}{(\alpha r)^{1+\frac{1}{p}}}\frac{1}{n^{r}}\Big)\bigg), 
\end{align}
 where the quantities ${\gamma_{n,p}^{(i)}=\gamma_{n,p}^{(i)}(\alpha,r,\beta)}$,  satisfy the inequality  ${|\gamma_{n,p}^{(i)}|\leq(14\pi)^{2}}$, \linebreak $i=1,2$.
Formula (\ref{Theorem1Ineq1}) follows from (\ref{for2}) and (\ref{normKern1}).

 To prove the second part of Theorem~\ref{theorem1}, according to the equality (\ref{for1}), for arbitrary $\varphi\in L_{p}$ we should find the function $\Phi(\cdot)=\Phi(\varphi,n;\cdot)\in L_{p}$, such that $E_{n}(\Phi)_{L_{p}}=E_{n}(\varphi)_{L_{p}}$ and for all $n\geq n_0(\alpha,r,p)$ the following equality holds
 \begin{align}\label{eq1}
 &\frac{1}{\pi}\left| \int\limits_{-\pi}^{\pi}(\Phi(t)-t_{n-1}^{*}(t)) P_{\alpha,r,\beta}^{(n)}(-t)dt  \right| 
  = e^{-\alpha n^{r}}n^{\frac{1-r}{p}} \left( 
\frac{\|\cos t\|_{p'}}{\pi^{1+\frac{1}{p'}}(\alpha r)^{\frac{1}{p}}}F^{\frac{1}{p'}}\Big(\frac{1}{2}, \frac{3-p'}{2}; \frac{3}{2}; 1\Big) \right.
 \notag \\
 &
 \left.+ 
 \gamma_{n,p}\Big(\Big(1+\frac{(\alpha r)^{\frac{p'-1}{p}}}{p'-1}\Big) \frac{1}{n^{\frac{1-r}{p}}}+\frac{(p)^{\frac{1}{p'}}}{(\alpha r)^{1+\frac{1}{p}}}\frac{1}{n^{r}}\Big)
   \right)
 E_{n}(\varphi)_{L_{p}}, \ \ \frac{1}{p}+\frac{1}{p'}=1,
 \end{align}
where $t_{n-1}^{*}$ is the polynomial of the best approximation of the order $n-1$ of the function $\Phi$ in the space $L_{p}$, $|\gamma_{n,p}|\leq (14\pi)^{2}$.

In this case for an arbitrary function $f\in C^{\alpha,r}_{\beta}L_{p}$, $1<p<\infty$, there exists a function $\Phi(\cdot)=\Phi(f^{\alpha,r}_{\beta}; \cdot)$, such that $E_{n}(\Phi)_{L_{p}}=E_{n}(f^{\alpha,r}_{\beta})_{L_{p}}$, and for  $n\geq n_0(\alpha,r,p)$ the formula \eqref{eq1} holds, where as function $\varphi$ we take the function $f^{\alpha,r}_{\beta}$.

Let us assume 
\begin{equation*}
\mathcal{F}\cdot)=\mathcal{J}^{\alpha,r}_{\beta}\Big(\Phi(\cdot)-\frac{a_{0}}{2}\Big),
\nonumber
\end{equation*}
where
\begin{equation*}
a_{0}=a_{0}(\Phi):=\frac{1}{\pi}\int\limits_{-\pi}^{\pi}\Phi(t)dt.
\nonumber
\end{equation*}
The function $F$ is the function,  which we have looked for, because $F\in C^{\alpha,r}_{\beta}L_{p}$ and
\begin{equation*}
E_{n}(\mathcal{F}^{\alpha,r}_{\beta})_{L_{p}}=E_{n}(\Phi-\frac{a_{0}}{2})_{L_{p}}=
E_{n}(\Phi)_{L_{p}}=E_{n}(f^{\alpha,r}_{\beta})_{L_{p}},
\nonumber
\end{equation*}
so  (\ref{repr}), (\ref{for1}),  (\ref{Theorem1Ineq1}) and (\ref{eq1}) imply (\ref{Theorem1Eq1}).

At last let us prove (\ref{eq1}). Let $\varphi\in L_{p}$, $1<p<\infty$.
Then as a function $\Phi(t)$ we consider the function
 \begin{equation}\label{Phi}
\Phi(t)=\|P_{\alpha,r, -\beta}^{(n)}\|_{p'}^{1-p'} |P_{\alpha,r, -\beta}^{(n)}(t)|^{p'-1}\mathrm{sign}(P_{\alpha,r, -\beta}^{(n)}(t)) E_{n}(\varphi)_{L_{p}}
\end{equation}

For this function
 \begin{align*}
\|\Phi\|_{p}&=\|P_{\alpha,r, -\beta}^{(n)}\|_{p'}^{1-p'} \||P_{\alpha,r, -\beta}^{(n)}|^{p'-1}\|_{p}    E_{n}(\varphi)_{L_{p}}
\notag \\
&=
\|P_{\alpha,r, -\beta}^{(n)}\|_{p'}^{1-p'} \|P_{\alpha,r, -\beta}^{(n)}\|_{p'}^{p'-1}    E_{n}(\varphi)_{L_{p}}=
E_{n}(\varphi)_{L_{p}}.
\end{align*}

Now we show that the polynomial $t_{n-1}^{*}$ of best approximation of order $n-1$ in the space $L_{p}$ of the function $\Phi(t)$ equals identically to zero: $t_{n-1}^{*}\equiv0$.

For any $t_{n-1}\in\tau_{2n-1}$  
\begin{equation*}
\int\limits_{0}^{2\pi}t_{n-1}(t) |\Phi(t)|^{p-1}\mathrm{sign}(\Phi(t))dt\!=\!
\| P_{\alpha,r, -\beta}^{(n)}\|_{p'}^{-1}(E_{n}(\varphi)_{L_{p}})^{p-1}\!\int\limits_{-\pi}^{\pi}t_{n-1}(t)P_{\alpha,r, -\beta}^{(n)}(t)dt
=0.
\end{equation*}

Then, according to Proposition 1.4.12 of the work \cite[p. 29]{Korn} we can make conclusion, that the polynomial $t_{n-1}^{*}\equiv0$ is the polynomial of the best approximation of the function $\Phi(t)$ in the space $L_{p}$, $1<p<\infty$.

For the function $\Phi(t)$ of the form \eqref{Phi} we can write 
\begin{align}\label{Phi_equal}
&\frac{1}{\pi}\int\limits_{-\pi}^{\pi} (\Phi(t)-t_{n-1}^{*}(t))P_{\alpha,r,\beta}^{(n)}(-t)dt
\notag \\
=&\frac{1}{\pi}\int\limits_{-\pi}^{\pi} \Phi(t)P_{\alpha,r,\beta}^{(n)}(-t)dt
=\frac{1}{\pi}\int\limits_{-\pi}^{\pi} \Phi(t)P_{\alpha,r,-\beta}^{(n)}(t)dt
\notag \\
=&\frac{1}{\pi}\|P_{\alpha,r, -\beta}^{(n)} \|_{p'}^{1-p'} E_{n}(\varphi)_{L_{p}}
\int\limits_{-\pi}^{\pi} |P_{\alpha,r,-\beta}^{(n)}(t)|^{p'}dt
=\frac{1}{\pi}\|P_{\alpha,r, -\beta}^{(n)} \|_{p'} E_{n}(\varphi)_{L_{p}}.
\end{align}
Thus from (\ref{normKern1})  and (\ref{Phi_equal}) we get \eqref{Theorem1Eq1}.
Theorem~\ref{theorem1} is proved.
\end{proof}

\begin{theorem}\label{theorem2}
Let $0<r<1$, $\alpha>0$, $\beta\in\mathbb{R}$, $n\in\mathbb{N}$. Then, for any 
 $f\in C^{\alpha,r}_{\beta}L_{1}$ and $n\geq n_0(\alpha,r,1)$ the following inequality holds:
 \begin{equation}\label{Theorem2Ineq1}
\|f(\cdot)-S_{n-1}(f;\cdot)\|_{C}\leq e^{-\alpha n^{r}}n^{1-r}\Big(
\frac{1}{\pi\alpha r}+\gamma_{n,1}\Big(\frac{1}{(\alpha r)^{2}}\frac{1}{n^{r}}+\frac{1}{n^{1-r}}\Big)\Big)E_{n}(f^{\alpha,r}_{\beta})_{L_{1}}.
 \end{equation}
 
 Moreover, for any function  $f\in C^{\alpha,r}_{\beta}L_{1}$ one can find a function ${\mathcal{F}x)=\mathcal{F}f;n,x)}$ in the set $C^{\alpha,r}_{\beta}L_{1}$, such that
  $E_{n}(\mathcal{F}^{\alpha,r}_{\beta})_{L_{1}}=E_{n}(f^{\alpha,r}_{\beta})_{L_{1}}$ 
and for $n>n_{0}(\alpha,r,1)$ the following equality holds  
 \begin{equation}\label{Theorem2Eq}
\|\mathcal{F}\cdot)-S_{n-1}(F;\cdot)\|_{C}= e^{-\alpha n^{r}}n^{1-r}\Big(
\frac{1}{\pi\alpha r}+\gamma_{n,1}\Big(\frac{1}{(\alpha r)^{2}}\frac{1}{n^{r}}+\frac{1}{n^{1-r}}\Big)\Big)E_{n}(f^{\alpha,r}_{\beta})_{L_{1}}.
 \end{equation}  
  
In (\ref{Theorem2Ineq1})  and (\ref{Theorem2Eq}) the quantity ${\gamma_{n,1}=\gamma_{n,1}(\alpha,r,\beta)}$ is such that ${|\gamma_{n,1}|\leq(14\pi)^{2}}$.
\end{theorem}

\begin{proof}[Proof of Theorem~\ref{theorem2}]
At the beginning let us show  that (\ref{Theorem2Ineq1}) holds.
Let $f\in C^{\alpha,r}_{\beta}L_{1}$. Then, according to (\ref{for1}) and (\ref{HolderIneq})
 \begin{equation}\label{ff1}
\|f(\cdot)-S_{n-1}(f;\cdot)\|_{C}= 
\frac{1}{\pi} \int\limits_{-\pi}^{\pi} (f^{\alpha, r}_{\beta}(t)-t_{n-1}^{*}(t))P_{\alpha,r,\beta}^{(n)}(x-t)dt 
\leq \frac{1}{\pi} \|P_{\alpha,r,\beta}^{(n)} \|_{\infty}E_{n}(f^{\alpha,r}_{\beta})_{L_{1}},
 \end{equation}  
where $t_{n-1}^{*}\in \tau_{2n-1}$ is the polynomial of the best approximation of the function $f^{\alpha,r}_{\beta}$ in the space $L_{1}$.

From formula (20) of the work \cite{SerdyukStepanyuk2017} (see also \cite{SerdyukStepanyuk2016}  and \cite{SerdyukStepanyuk2018}) for arbitrary $r\in(0,1)$, $\alpha>0$, $\beta\in \mathbb{R}$, $n\in\mathbb{N}$, $n\geq n_{0}(\alpha, r,1)$ it follows that
 \begin{equation}\label{ff2}
 \frac{1}{\pi} \|P_{\alpha,r,\beta}^{(n)} \|_{\infty}
 =e^{-\alpha n^{r}}n^{1-r}\Big(\frac{1}{\alpha r\pi}+\gamma_{n,1}\Big(\frac{1}{(\alpha r)^{2}n^{r}}+\frac{1}{n^{1-r}} \Big) \Big),
 \end{equation}  
where the quantity ${\gamma_{n,1}=\gamma_{n,1}(\alpha,r,\beta)}$ is such that ${|\gamma_{n,1}|\leq(14\pi)^{2}}$.

It is clear, that from  $P_{\alpha,r,\beta}^{(n)}\in C$ it follows that the norm $\|P_{\alpha,r,\beta}^{(n)}\|_{\infty}$ in (\ref{ff1}) and (\ref{ff2}) can be substituted by $\|P_{\alpha,r,\beta}^{(n)}\|_{C}$.

Combining formulas (\ref{ff1}) and (\ref{ff2}), we get (\ref{Theorem2Ineq1}).

To prove the second part of Theorem~\ref{theorem2}
we need for any function $\varphi\in L_{1}$ to find the function $\Phi(\cdot)=\Phi(\varphi, \cdot)\in L_{1}$, such that
$E_{n}(\Phi)_{L_{1}}=E_{n}(\varphi)_{L_{1}}$ and for all $n\geq n_{0}(\alpha,r,1)$ the following equality holds
\begin{align}\label{th2Eq1}
&\frac{1}{\pi}\left|\int\limits_{-\pi}^{\pi}\left(\Phi(t) -t_{n-1}^{*}(t)\right)P_{\alpha,r,\beta}^{(n)}(-t)dt\right| \notag \\
&=e^{-\alpha n^{r}}n^{1-r} 
\Big(
\frac{1}{\pi\alpha r}+\gamma_{n,1}\Big(\frac{1}{(\alpha r)^{2}}\frac{1}{n^{r}}+\frac{1}{n^{1-r}}\Big)\Big)E_{n}(\varphi)_{L_{1}},
\end{align}
where $t_{n-1}^{*}$ is the polynomial of the best approximation of order $n-1$ of the function $\Phi$ in the space $L_{1}$ and ${|\gamma_{n,1}|\leq(14\pi)^{2}}$.

In this case for any function $f\in C^{\alpha,r}_{\beta}L_{1}$ there exists a function $\Phi(\cdot)=\Phi(f^{\alpha,r}_{\beta};\cdot)$, such that $E_{n}(\Phi)_{L_{1}}=E_{n}(f^{\alpha,r}_{\beta})$, and for $n\geq n_{0}(\alpha,r,1)$ the formula (\ref{th2Eq1})  holds, where as function $\varphi$ we will take the function $f^{\alpha,r}_{\beta}$.

Let us consider the function
\begin{equation*}
\mathcal{F}\cdot)=\mathcal{J}^{\alpha,r}_{\beta}(\Phi(\cdot)-\frac{a_{0}}{2}),
\end{equation*}
where 
\begin{equation*}
a_{0}=a_{0}(\Phi):=\frac{1}{\pi}\int\limits_{-\pi}^{\pi}\Phi(t)dt.
\end{equation*}

The function $F$ is the function, which we look for, because
$F\in C^{\alpha,r}_{\beta}L_{1}$ and 
\begin{equation*}
E_{n}(\mathcal{F}^{\alpha,r}_{\beta})_{L_{1}}=E_{n}(\Phi-\frac{a_{0}}{2})_{L_{1}}=
E_{n}(\Phi)_{L_{1}}=E_{n}(f^{\alpha,r}_{\beta})_{L_{1}},
\nonumber
\end{equation*}
and on the basis (\ref{repr}),  (\ref{for1}), (\ref{Theorem2Ineq1})  and  (\ref{th2Eq1}) the formula (\ref{Theorem2Eq}) holds.

Let us prove (\ref{th2Eq1}). Let $t^{*}$ be the point from the interval $T=\Big[\frac{\pi(1-\beta)}{2n}, \ 2\pi+\frac{\pi(1-\beta)}{2n} \Big)$, where the function $|P_{\alpha,r,-\beta}^{(n)}|$  attains its largest value, i.e.,
\begin{equation*}
|P_{\alpha,r,-\beta}^{(n)}(t^{*})|=\| P_{\alpha,r,-\beta}^{(n)}\|_{C}=
\| P_{\alpha,r,\beta}^{(n)}\|_{C}.
\nonumber
\end{equation*}

Let put $\Delta_{k}^{n}:=\Big[\frac{(k-1)\pi}{n}+\frac{\pi(1-\beta)}{2n}, \frac{k\pi}{n}+\frac{\pi(1-\beta)}{2n} \Big)$, $k=1,...,2n$.
By  $k^{*}$  we denote the number, such that $t^{*}\in \Delta_{k^{*}}^{n}$.
Taking into account, that function $P_{\alpha,r,-\beta}^{(n)}$ is absolutely continuous, so for arbitrary $\varepsilon>0$ there exists a segment $\ell^{*}=[\xi^{*}, \xi^{*}+\delta]\subset \Delta_{k^{*}}^{n}$, such that for arbitrary $t\in \ell^{*}$ the following inequality holds
${|P_{\alpha,r,-\beta}^{(n)}(t)|>\| P_{\alpha,r,\beta}^{(n)}\|_{C}-\varepsilon}$.
It is clear that $\mathrm{mes}\, \ell^{*}=|\ell^{*}|=\delta<\frac{\pi}{n}$.

For arbitrary $\varphi\in L_{1}$ and $\varepsilon>0$ we consider the function $\Phi_{\varepsilon}(t)$, which on the segment $T$ is defined with a help of equalities

\begin{equation*}
\Phi_{\varepsilon}(t)=
\begin{cases}
E_{n}(\varphi)_{L_1}\frac{1-\varepsilon(2\pi-\delta)}{\delta}\mathrm{sign}\cos \Big( nt+\frac{\beta\pi}{2}\Big), & t\in \ell^{*}, \\
E_{n}(\varphi)_{L_1} \varepsilon \  \mathrm{sign}\cos\Big( nt+\frac{\beta\pi}{2}\Big), &
t\in \mathrm{T}\setminus \ell^{*}.
  \end{cases}
\nonumber
\end{equation*}

For the function $\Phi_{\varepsilon}(t)$ for arbitrary small values of $\varepsilon>0$  $(\varepsilon\in(0, \frac{1}{2\pi}))$ the following equality holds
\begin{align}\label{Phi_L1}
 \|\Phi_{\varepsilon}\|_{1}&=
E_{n}(\varphi)_{L_1}\frac{1-\varepsilon(2\pi-\delta)}{\delta}
\int\limits_{\ell^{*}}\Big| \mathrm{sign}\cos\Big( nt+\frac{\beta\pi}{2}\Big) \Big| dt \notag \\
&+
E_{n}(\varphi)_{L_1}\varepsilon
\int\limits_{\mathrm{T}\setminus \ell^{*}}\Big| \mathrm{sign}\cos\Big( nt+\frac{\beta\pi}{2}\Big) \Big| dt
\notag \\
&= 
E_{n}(\varphi)_{L_1}\left(\frac{1-\varepsilon(2\pi-\delta)}{\delta}\delta+\varepsilon(2\pi-\delta)   \right)=E_{n}(\varphi)_{L_1}.
\end{align}

It should be noticed, that
\begin{equation}\label{sign}
\mathrm{sign} \Phi_{\varepsilon}(t)=\mathrm{sign}\cos\Big( nt+\frac{\beta\pi}{2}\Big).
\end{equation}

Since for arbitrary trigonometric polynomial $t_{n-1}\in \tau_{2n-1}$
\begin{equation*}
\int\limits_{0}^{2\pi}t_{n-1}(t)\mathrm{sign}\cos\Big( nt+\frac{\beta\pi}{2}\Big) dt=0,
\nonumber
\end{equation*}

so, taking into account (\ref{sign})
\begin{equation*}
\int\limits_{0}^{2\pi}t_{n-1}(t)\mathrm{sign}\Big(\Phi_{\varepsilon}(t)-0 \Big) dt=0,
\ \ \ t_{n-1}\in\tau_{2n-1}.
\nonumber
\end{equation*}

According to  Proposition 1.4.12 of the work \cite[p.29]{Korn} the polynomial $t_{n-1}^{*}\equiv0$ is a polynomial of the best approximation of the function $\Phi_{\varepsilon}$ in the metric of the space $L_{1}$, i.e., $E_{n}(\Phi_{\varepsilon})_{L_1}=\| \Phi_{\varepsilon}\|_{1}$, so  (\ref{Phi_L1}) yields
$E_{n}(\Phi_{\varepsilon})_{L_1}=E_{n}(\varphi)_{L_1}$.

Moreover, for the function $\Phi_{\varepsilon}$
\begin{align}\label{form_eq1}
&\frac{1}{\pi}\int\limits_{-\pi}^{\pi}(\Phi_{\varepsilon}(t)-t_{n-1}^{*}(t))P_{\alpha,r,\beta}^{(n)}(-t)dt
 =
\frac{1}{\pi}\int\limits_{-\pi}^{\pi}\Phi_{\varepsilon}(t)P_{\alpha,r,-\beta}^{(n)}(t)dt \notag \\
=&\frac{1-\varepsilon(2\pi-\delta)}{\pi\delta}
E_{n}(\varphi)_{L_1}
\int\limits_{\ell^{*}} \mathrm{sign}\cos\Big( nt+\frac{\beta\pi}{2}\Big) P_{\alpha, r, -\beta}^{(n)}(t) dt \notag \\
+&
\frac{\varepsilon}{\pi}E_{n}(\varphi)_{L_1}
\int\limits_{\mathrm{T}\setminus \ell^{*}} \mathrm{sign}\cos\Big( nt+\frac{\beta\pi}{2}\Big) P_{\alpha, r, -\beta}^{(n)}(t) dt.
\end{align}

Taking into account, that $ \mathrm{sign} \Phi_{\varepsilon}(t)=(-1)^{k}$, $t\in \Delta_{k}^{(n)}, \  k=1,..., 2n$, and also the embedding $\ell^{*}\subset \Delta_{k^{*}}^{(n)}$, we get
\begin{align}\label{form_eq2}
&\left|\frac{1-\varepsilon(2\pi-\delta)}{\pi\delta}
E_{n}(\varphi)_{L_1}
\int\limits_{\ell^{*}} \mathrm{sign}\cos\Big( nt+\frac{\beta\pi}{2}\Big) P_{\alpha, r, -\beta}^{(n)}(t) dt \right| \notag \\
=
&\left|(-1)^{k^{*}}\frac{1-\varepsilon(2\pi-\delta)}{\pi\delta}
E_{n}(\varphi)_{L_1}
\int\limits_{\ell^{*}}  P_{\alpha, r, -\beta}^{(n)}(t) dt \right| \notag \\
\geq& 
\frac{1-\varepsilon(2\pi-\delta)}{\pi}
E_{n}(\varphi)_{L_1}
\left( \|P_{\alpha, r, \beta}^{(n)}\|_{C}-\varepsilon\right)
\notag \\
>&
\frac{1-2\pi\varepsilon}{\pi}
E_{n}(\varphi)_{L_1}
\left( \|P_{\alpha, r, \beta}^{(n)}\|_{C}-\varepsilon\right) \notag \\
=& 
\frac{1}{\pi}
E_{n}(\varphi)_{L_1}
\left( \|P_{\alpha, r, \beta}^{(n)}\|_{C}-
2\pi\varepsilon \|P_{\alpha, r, \beta}^{(n)}\|_{C}-\varepsilon+2\pi\varepsilon^{2}\right)\notag \\
>&
E_{n}(\varphi)_{L_1} \left(\frac{1}{\pi} \|P_{\alpha, r, \beta}^{(n)}\|_{C}-
\varepsilon \Big(2 \|P_{\alpha, r, \beta}^{(n)}\|_{C}+\frac{1}{\pi}\Big) \right).
\end{align}

Also, it is not hard to see that 
\begin{equation}\label{form_eq3}
\left| \frac{\varepsilon}{\pi}E_{n}(\varphi)_{L_1}
\int\limits_{\mathrm{T}\setminus \ell^{*}} \mathrm{sign}\cos\Big( nt+\frac{\beta\pi}{2}\Big) P_{\alpha, r, -\beta}^{(n)}(t) dt\right| \leq\frac{\varepsilon}{\pi}E_{n}(\varphi)_{L_1} \| P_{\alpha, r, \beta}^{(n)}\|_{C}.
\end{equation}

Formulas (\ref{form_eq1})--(\ref{form_eq3}) yield the following inequality
\begin{align}\label{form_eq4}
& \left| \int\limits_{-\pi}^{\pi}\frac{1}{\pi}(\Phi_{\varepsilon}(t)-t_{n-1}^{*}(t)) P_{\alpha, r, \beta}^{(n)}(-t)dt \right| \notag \\
>&
E_{n}(\varphi)_{L_1}\left(\frac{1}{\pi} \| P_{\alpha, r, \beta}^{(n)}\|_{C}
-\varepsilon\Big(\Big(2+\frac{1}{\pi} \Big) \| P_{\alpha, r, \beta}^{(n)}\|_{C}+\frac{1}{\pi}\Big)\right).
\end{align}

Let us show, that on basis of the results of the work \cite{SerdyukStepanyuk2018}, the estimate (\ref{ff2}) can be improved, if we decrease the diapason for $|\gamma_{n,1}|$.

Formulas (34), (50)--(52) of the work \cite{SerdyukStepanyuk2018}, and also Remark 1 from \cite{SerdyukStepanyuk2018} allow us to write that for any $n\in\mathbb{N}$
\begin{equation}\label{form_eq5}
\|P_{\alpha, r, \beta}^{(n)}\|_{\infty}=
\|P_{\alpha, r, n}\|_{\infty}\Big(1+\delta_{n}^{(1)}\frac{M_{n}}{n} \Big),
\end{equation}
where
\begin{equation*}
P_{\alpha, r, n}(t)
:=\sum\limits_{k=0}^{\infty}e^{-\alpha(k+n)^{r}}e^{ikt},
\end{equation*}
\begin{equation*}
M_{n}:
=\sup\limits_{t\in\mathbb{R}}\frac{|P'_{\alpha,r,n}(t)|}{|P_{\alpha,r,n}(t)|},
\end{equation*}
and for $\delta_{n}^{(1)}=\delta_{n}^{(1)}(\alpha,r,\beta)$ the following  estimate takes place $|\delta_{n}^{(1)}|\leq 5\sqrt{2}\pi$.

Then, as it follows from the estimates (87) and (99) of the work \cite{SerdyukStepanyuk2018} for ${n\geq n_{0}(\alpha,r,1)}$
\begin{equation}\label{form_eq8}
\|P_{\alpha, r, n}\|_{\infty}=
\frac{e^{-\alpha n^{r}}}{\alpha r}n^{1-r}\left(1+\theta_{\alpha,r,n}\Big(\frac{1-r}{\alpha r n^{r}}+\frac{\alpha r}{n^{1-r}} \Big) \right), \ \ \ |\theta_{\alpha,r,n}|\leq \frac{14}{13}
\end{equation}
and
\begin{equation}\label{form_eq9}
M_{n}\leq \frac{784\pi^{2}}{117}\Big( \frac{n^{1-r}}{\alpha r}+\alpha rn^{r}\Big).
\end{equation}

Combining formulas (\ref{form_eq5})--(\ref{form_eq9}) we obtain that for $n\geq n_{0}(\alpha,r,1)$
\begin{align}\label{form_eq10}
\frac{1}{\pi}\|P_{\alpha, r, \beta}^{(n)}\|_{\infty}&=
\frac{e^{-\alpha n^{r}}}{\alpha r \pi}n^{1-r}\left(1+ \theta_{\alpha,r,n}\Big(\frac{1-r}{\alpha r n^{r}}+\frac{\alpha r}{n^{1-r}} \Big)  \right)\Big(1+\delta_{n}^{(1)}\frac{M_{n}}{n}\Big) \notag \\
&=
e^{-\alpha n^{r}}n^{1-r}\Big(\frac{1}{\alpha r \pi}+\gamma_{n,1}\Big(\frac{1}{\alpha rn^{1-r}}+\frac{1}{n^{1-r}} \Big) \Big),
\end{align}
where 
\begin{equation}\label{form_eq11}
|\gamma_{n,1}|\leq \frac{1}{\pi}\left( \frac{14}{13}+ \frac{784\pi^{2} 5\sqrt{2}\pi}{117}+
\frac{14\cdot5\sqrt{2}\pi\cdot784\pi^{2}}{13\cdot117\cdot14} \right) 
=
\frac{14}{13\pi}\Big(1+\frac{3920\sqrt{2}\pi^{3}}{117} \Big).
\end{equation}

Let us  choose $\varepsilon$ small enough, that
\begin{equation}\label{form_eq12}
\varepsilon<\frac{\left((14\pi)^{2}-\frac{14}{13\pi}\left(1+  \frac{3920\sqrt{2}\pi^{3}}{117} \right) \right) e^{-\alpha n^{r}}n^{1-r} (\frac{1}{\alpha r n^{r}}+\frac{\alpha r}{n^{1-r}})}{(2+\frac{1}{\pi})\|P_{\alpha, r, \beta}^{(n)}\|_{\infty}+\frac{1}{\pi}}
\end{equation}
 and for this $\varepsilon$ we put 
\begin{equation}\label{form_eq13}
\Phi(t)=\Phi_{\varepsilon}(t).
\end{equation}

The function $\Phi(t)$ is the function, which we looked for, because 
 ${E_{n}(\Phi)_{L_1}=E_{n}(\varphi)_{L_1}}$ and according to (\ref{form_eq4}), (\ref{form_eq10})--(\ref{form_eq12}) for $n\geq n_{0}(\alpha,r,1)$
 \begin{align}\label{form_eq14}
&\left| \frac{1}{\pi}(\Phi(t)-t_{n-1}^{*}(t))P_{\alpha,r,\beta}^{(n)}(-t)dt\right| \notag \\
& >\!E_{n}(\varphi)_{L_1} \!\!\left(\frac{1}{\pi}\|P_{\alpha,r,\beta}^{(n)} \|_{C} \!-\!
\left(\!\!(14\pi)^{2}\!-\!\frac{14}{13\pi}\Big(1\!+ \! \frac{3920\sqrt{2}\pi^{3}}{117}
 \Big)\!\! \right)\!e^{-\alpha n^{r}}n^{1-r}\Big(\frac{1}{\alpha r n^{r}}\!+\!\frac{\alpha r}{n^{1-r}} \Big)\!\!
\right) \notag \\
&\geq e^{-\alpha n^{r}}n^{1-r}\left( \frac{1}{\alpha r\pi}-(14\pi)^{2}   \Big(\frac{1}{\alpha r n^{r}}+\frac{\alpha r}{n^{1-r}}\Big) \right) E_{n}(\varphi)_{L_1}.
\end{align}
 
 Formulas (\ref{form_eq14}), (\ref{ff1}) and (\ref{ff2}) imply (\ref{th2Eq1}).
Theorem~\ref{theorem2} is proved.
\end{proof}

It should be noticed, that inequalities (\ref{Theorem1Ineq1}) and (\ref{Theorem2Ineq1}) were announced in the work \cite{SerdyukStepanyuk2018Bul}. There it was also mentioned that that estimates (\ref{Theorem1Ineq1}) and (\ref{Theorem2Ineq1}) are asymptotically best possible on the classes $C^{\alpha,r}_{\beta,p}$, $1\leq p<\infty$.

If $f\in C^{\alpha,r}_{\beta,p}$, then $\|f^{\alpha,r}_{\beta}\|_{p}\leq 1$, and 
$E_{n}(f^{\alpha,r}_{\beta})_{L_{p}}\leq 1$, $1\leq p<\infty$.
 Considering the least upper bounds of both sides of inequality (\ref{Theorem1Ineq1})  over the classes $C^{\alpha,r}_{\beta,p}$, $1<p<\infty$, we  arrive at the inequality
\begin{align}\label{Theorem1Ineq1Remark}
{\cal E}_{n}(C^{\alpha,r}_{\beta,p})_{C}&\leq e^{-\alpha n^{r}}n^{\frac{1-r}{p}}\bigg(\frac{\|\cos t\|_{p'}}{\pi^{1+\frac{1}{p'}}(\alpha r)^{\frac{1}{p}}}F^{\frac{1}{p'}}\Big(\frac{1}{2}, \frac{3-p'}{2}; \frac{3}{2}; 1\Big)\notag \\
&
+
\gamma_{n,p}\Big(\Big(1+\frac{(\alpha r)^{\frac{p'-1}{p}}}{p'-1}\Big) \frac{1}{n^{\frac{1-r}{p}}}+\frac{(p)^{\frac{1}{p'}}}{(\alpha r)^{1+\frac{1}{p}}}\frac{1}{n^{r}}\Big)\bigg)E_{n}(f^{\alpha,r}_{\beta})_{L_{p}}, \ \ \frac{1}{p}+\frac{1}{p'}=1.
\end{align}

 Comparing this relation with the estimate of Theorem 4  from  \cite{SerdyukStepanyuk2017} (see also \cite{SerdyukStepanyuk2018}),
 we conclude that   inequality (\ref{Theorem1Ineq1}) on the classes $C^{\alpha,r}_{\beta,p}$, $1<p<\infty$, is asymptotically best possible.

  In the same way, the asymptotic sharpness of the estimate (\ref{Theorem2Ineq1}) on  the classes $C^{\alpha,r}_{\beta,1}$ follows from comparing 
 inequality
\begin{equation}\label{Theorem1Ineq2Remark}
{\cal E}_{n}(C^{\alpha,r}_{\beta,p})_{C} \leq e^{-\alpha n^{r}} n^{1-r}\Big(
\frac{1}{\pi\alpha r}+\gamma_{n,1}\Big(\frac{1}{(\alpha r)^{2}}\frac{1}{n^{r}}+\frac{1}{n^{1-r}}\Big)\Big)E_{n}(f^{\alpha,r}_{\beta})_{L_{1}}
\end{equation}  
   and formula (18) from  \cite{SerdyukStepanyuk2018}.


\section{Acknowledgements}

The second author is supported by the Austrian Science Fund FWF project F5506-N26 (part of the Special Research Program (SFB) ``Quasi-Monte Carlo Methods: Theory and Applications'') and partially is supported by grant of NAS of Ukraine for  groups of young scientists (project No16-10/2018).


\end{document}